\documentclass{amsart}

\usepackage{amsmath, amsthm, amsfonts}
\usepackage{amssymb}
\usepackage[utf8]{inputenc}
\usepackage{latexsym}
\usepackage{verbatim}
\usepackage{url}
\usepackage{textcomp}
\usepackage[all,cmtip]{xy}
\usepackage{color}
\usepackage{hyperref}
\usepackage{tabularx}
\input{xypic}
\usepackage{enumerate}
\usepackage{amssymb}
\usepackage{array}

\DeclareMathAlphabet{\mathfr}{U}{euf}{m}{n}

\newtheorem{theorem}{Theorem}
\newtheorem*{theorem*}{Theorem}

\newtheorem{proposition}[theorem]{Proposition}
\newtheorem{corollary}[theorem]{Corollary}

\theoremstyle{remark}
\newtheorem{remark}[theorem]{Remark}
\newtheorem{definition}[theorem]{Definition}

\newcommand{\Q}{\mathbb Q}
\newcommand{\Qbar}{{\overline{\mathbb Q}}}

\newcommand{\Gal}{\mathrm{Gal}}

\newcommand{\Z}{\mathbb Z}

\newcommand{\C}{\mathbb C}

\newcommand{\GL}{\mathrm{GL}}
\newcommand{\M}{\mathrm{M }}

\newcommand{\End}{\operatorname{End}}

\newcommand{\Frob}{\operatorname{Frob}}

\newcommand{\Aut}{\operatorname{Aut}}
\newcommand{\Oo}{\mathcal O}

\newcommand{\p}{\mathfrak{p}}
\newcommand{\PP}{\mathfrak{P}}
\newcommand{\bb}{\mathfrak{b}}
\newcommand{\Ll}{\mathfrak{L}}
\newcommand{\Res}{\operatorname{Res}}
\newcommand{\Tr}{\operatorname{Tr}}

\newcommand{\Nm}{\operatorname{Nm}}

\newcommand{\Mm}{\mathfrak{M}}

\usepackage{url}

\begin{document}
\title[Ordinary primes for some varieties with extra endomorphisms]{Ordinary primes for some varieties\\ with extra endomorphisms}
\author{Francesc Fit\'e}


\address{Departament de matem\`atiques i inform\`atica and Centre de recerca matem\`atica,
Universitat de Barcelona,
Gran via de les Corts Catalanes 585, 08007 Barcelona, Catalonia}
\email{ffite@ub.edu}
\urladdr{http://www.ub.edu/nt/ffite/}

\date{\today}

\begin{abstract} Let $A$ be an abelian variety defined over a number field and of dimension $g$. When $g\leq 2$, by the recent work of Sawin, we know the exact (nonzero) value of the density of the set of primes which are ordinary for $A$. In higher dimension very little is known. We show that if $g=3$ and $A$ has multiplication by an imaginary quadratic field $E$, then there exists a nonzero density set of ordinary primes for~$A$. We reach the same conclusion if $g=4$ and the pair $(A,E)$ has signature $(2,2)$. We also obtain partial results when $g=3$ and $A$ has multiplication by a totally real cubic field. We show that our methods also apply to certain abelian varieties of Albert type IV of higher dimension. These results are derived from an extended version of the $\ell$-adic methods of Katz, Ogus, and Serre in the presence of extra endomorphisms. 
\end{abstract}

\maketitle

\section{Introduction} Throughout this note let $k$ denote a number field and $A$ be an abelian variety defined over $k$ of dimension $g\geq 1$. Let $\p$ be a prime of good reduction for $A$ and choose a rational prime $\ell$ not divisible by $\p$. We will denote by $P_\p(A,T)$ the characteristic polynomial of a Frobenius element $\Frob_\p$ acting on the rational $\ell$-adic Tate module $V_\ell(A)$ of $A$. The polynomial $P_\p(A,T)$ is monic of degree $2g$, has integer coefficients, and does not depend on the choice of $\ell$. 

We will denote by $A_\p$ the reduction of $A$ modulo $\p$. If $p$ is the rational prime below $\p$, let $v$ denote an extension to the ring of algebraic integers $\overline \Z$ of the $p$-adic valuation. Let us denote by $u_\p(A)$ the number of roots of $P_\p(A,T)$ in $\overline\Z$ which have $v$-adic valuation $0$. We say that $\p$ is \emph{ordinary} for $A$ if $u(A_\p)$ coincides with the dimension $g$, or if equivalently the coefficient of $T^g$ in $P_\p(A,T)$ is not divisible by~$p$.

When $A$ is an elliptic curve, it is well known (see \cite{Ser89}) that the set of ordinary primes has density $1/2$ or $1$, depending on whether $A$ has complex multiplication or not. Sawin \cite[Thm. 3]{Saw16} has determined the density of the set of ordinary primes when $A$ is an abelian surface: it only attains the values $1/4$, $1/2$ or $1$. Note that this density had been previously known to be positive (see \cite[p. 370--372]{Ogu82}). More in general, it is a conjecture often attributed to Serre that for any abelian variety the set of ordinary primes has a nonzero density.  

It is nonetheless worth noting that the conjecture is not known in many cases beyond dimension 2. The abundance of ordinary primes is known if $A$ has potential complex multiplication (see Remark~\ref{remark: CM}), and it has also been proven in cases where $\End(A_\Qbar)= \Z$ and the Mumford--Tate group is in some sense small (see the examples considered in \cite{Mum69}). For this we refer to the results by Noot \cite[Thm. 2.2]{Noo95} and Pink \cite[Thm. 7.1]{Pin98}. As we will see below, the conjecture also follows from the work of \cite{Suh20} for certain modular abelian varieties.

The original motivation for this note was to investigate the density of the set of ordinary primes for a nongeneric abelian threefold, that is, an abelian threefold $A$ for which $\End(A_\Qbar)\not = \Z$. If $A$ decomposes up to $\Qbar$-isogeny, then the existence of a nonzero density set of ordinary primes for $A$ follows from the additivity of the function $u_\p$ together with the results of Sawin and what is known for elliptic curves. As mentioned above, the existence of such a set when $A$ has potential complex multiplication is well known. Albert's classification (see for example \cite[\S21]{Mum74}) leaves us with two possibilities to consider: the case in which $A$ has multiplication by an imaginary quadratic field and the case in which $A$ has real multiplication by a totally real cubic field. Our approach was fully successful in the first case.

\begin{theorem}\label{theorem: main}
Let $A$ be an abelian threefold defined over a number field $k$. Suppose that $\End(A_\Qbar)\otimes \Q$ contains an imaginary quadratic field. Then there exists a set of nonzero density of primes of $k$ ordinary for $A$.
\end{theorem}

In particular, every Picard curve defined over a number field possesses a nonzero density set of ordinary primes. In the case of real multiplication, we were only able to ensure in general the existence of a nonzero density set of ``half ordinary" primes.

\begin{theorem}\label{theorem: main2}
Let $A$ be an abelian variety defined over a number field $k$ of dimension $g\geq 1$. Suppose that $\End(A_\Qbar)\otimes \Q$ contains a Galois number field of degree $g$. Then there exists a set of nonzero density of primes $\p$ of $k$ for which $u_\p(A)$ is at least $\lceil g/2\rceil$.
\end{theorem}

The above result is essencially contained in \cite[Thm. 4.1.1]{Suh20}. We will present a proof of it in \S\ref{section: main2} for the reader's convenience. Under an additional hypothesis of linear disjointness over $k$ between the field giving the real multiplication and the field of definition of this multiplication, we can ensure abundance of ordinary primes. 

\begin{theorem}\label{theorem: main2p}
Let $A$ be an abelian threefold defined over a number field $k$. Suppose that there exists a finite extension $K/k$ and a cubic real field $E$ such that $\End(A_K)\otimes \Q$ contains $E$ and $E$ is not contained in $K$. Then there exists a set of nonzero density of primes of $k$ ordinary for $A$.
\end{theorem}

In the above theorem we restrict to dimension 3 for simplicity. One can certainly find numerical conditions on the dimension of $A$ to derive the same conclusion in other cases. In particular, the theorem holds for the modular abelian varieties $A_f$ attached by Eichler and Shimura to (non CM) classical newforms $f$ of weight $2$ and cubic coefficient field.
In dimension $g=4$, we were able to prove an analogue of Theorem~\ref{theorem: main} under the assumption of an additional hypothesis (we refer to Definition~\ref{definition: signature} for the notion of signature).

\begin{theorem}\label{theorem: main3}
Let $A$ be an abelian fourfold defined over a number field $k$. Suppose that $\End(A_\Qbar)\otimes \Q$ contains an imaginary quadratic field and that the pair $(A,E)$ has signature $(2,2)$. Then there exists a set of nonzero density of primes of $k$ ordinary for~$A$.
\end{theorem}

As any quaternion algebra over $\Q$ contains an imaginary quadratic field, the previous theorem applies to abelian fourfolds $A$ such that $\End(A_\Qbar)\otimes \Q$ is isomorphic to a quaternion algebra over $\Q$ (such varieties necessarily have signature $(2,2)$; see Remark \ref{remark: quatabfourfold}). The paper concludes with an incursion to higher dimension, where we study abelian varieties which in some sense exhibit a behavior typical of lower dimension.

\begin{corollary}\label{corollary: main 5}
Let $A$ be an abelian variety defined over a number field $k$ of dimension $g\geq 1$ such that $\End(A_\Qbar)\otimes \Q$ contains a simple subalgebra $D$ with center $K$ and Schur index $t$. Suppose that $K$ is an imaginary quadratic field, and that either:
\begin{enumerate}[i)]
\item $t=g/2$ and $g$ divides $4$; or
\item $t=g/3$ and $g$ divides $9$; or
\item $t=g/4$, $g$ divides 16, and $(A,K)$ has signature $(g/2,g/2)$.
\end{enumerate}
Then there exists a set of nonzero density of primes of $k$ ordinary for~$A$.
\end{corollary} 

The corollary applies for example to some abelian varieties of Albert type IV. Note however that by \cite[Thm. 5]{Shi63}, when $t=g/2$ and $g$ divides $4$ then $D$ is strictly contained in $\End(A_\Qbar)\otimes \Q$.

Some sort of $p$-ordinary hypothesis is a common assumption in results concerning the construction of $p$-adic $L$-functions or $p$-adic families (notably in Hida theory). Related to this, Wiles \cite[pp. 571-572]{Wil88} gave an argument ensuring the existence of a nonzero density set of primes~$\p$ at which a Hilbert newform $f$ is ordinary at \emph{some} prime of its coefficient field above $\Nm(\p)$, as long as $f$ is defined over a field of degree at most~2 (but note that for $A_f$ be ordinary it does not suffice to ensure $f$ be ordinary only at \emph{some} prime above $\Nm(\p)$). In any case, $p$-adic constructions are local in nature, and thus the \emph{abundance} of ordinary primes is generally not an essential part of the problem.
There are, however, some exceptions: results about Galois representations where the abundance of ordinary primes does appear as a technical assumption. As a matter of example, we mention the automorphic lifting theorem \cite[Thm. 8.5.2]{BCGP18}, leading to the meromorphicity of the Hasse-Weil $L$-function of a generic abelian surface defined over a totally real field.

In the remainder of this note we will be concerned with the proof of the four theorems and the corollary stated above. We will see that the methods of Katz, Ogus, and Serre (successfully employed in \cite{Ogu82} to derive the existence of a nonzero density set of ordinary primes for an arbitrary abelian surface) can be applied to the situations considered thanks to the presence in $A$ of \emph{extra} multiplications by a number field $E$. The key point is to replace the $\ell$-adic representations $\varrho_{A,\ell}$ employed in the original argument by the $\lambda$-adic representations $\varrho_{A,\lambda}$, studied by Serre, Ribet, and Zarhin, among others. In order to prove Theorem \ref{theorem: main3}, we use a theorem of Tate to describe the infinity type of the Hecke characters $\det(\varrho_{A,\lambda})$ in terms of the action of $E$ on the tangent space of~$A$, which may be of independent interest (see Proposition \ref{proposition: RibetShimura}).

\subsection*{Notation and terminology.} As we have been doing so far, we will often abuse the terminology by simply calling ``primes of $F$'' the nonzero prime ideals of the ring of integers of a number field $F$. If $A$ is an abelian variety defined over a field~$F$, then by $\End(A)$ we mean the ring of endomorphisms of $A$ defined over~$F$. If $F'/F$ is a field extension, we denote by $A_{F'}$ the base change of $A$ from $F$ to $F'$. Given a number field $E$ and $\PP$ a prime of $E$, we will denote by $v_\PP$ the associated $\PP$-adic valuation. Once the prime $\PP$ has been chosen, we will usually denote by $v$ the extension of $v_\PP$ to $\overline\Z$ obtained by choosing a prime ideal of $\overline\Z$ lying above $\PP$.

\section{Galois representation theoretic input} Let $E$ be a number field and denote by $E_\lambda$ its completion at a prime $\lambda$ of its ring of integers $\Oo_E$. Denote by $\ell$ the rational prime below $\lambda$. Let $V$ be a vector space over $E_\lambda$ of dimension $d$. Denote by $G_k$ the absolute Galois group of $k$ and let $\varrho: G_k\rightarrow \Aut(V)$ be a continuous representation unramified outside a finite set $S$ of primes of $k$. Let $|\cdot|$ denote the complex absolute value and for a prime $\p$ of $k$ let $\Nm(\p)$ denote its absolute norm. Given a prime power $q$, we say that an algebraic number $\alpha$ is a $q$-Weil number if $|\iota(\alpha)|=\sqrt q$ for every embedding $\iota\colon \Q(\alpha)\rightarrow \C$. We will say that $\varrho$ is:
\begin{enumerate}[i)] 
\item \emph{integral} (or \emph{$\Oo_E$-integral} if we want to be more precise) if for every $\p$ outside $S$ the characteristic polynomial of $\varrho(\Frob_\p)$ lies in $\Oo_E[T]\subseteq E_\lambda[T]$; 
\item \emph{of weight $w$}, for an integer $w$,  if it is integral and, for every $\p$ outside $S$, every root $\alpha$ of the characteristic polynomial of $\varrho(\Frob_\p)$ is a $\Nm(\p)^w$-Weil number. 
\end{enumerate}

Below we will use the following notations without further mention. If $k'/k$ is a finite extension of $k$, we will denote by $S'$ the set of primes of $k'$ lying above~$S$. If $\p$ is a prime of $k$ outside $S$, we denote by $a_\p$ the trace $\Tr(\varrho(\Frob_\p))$. If $\PP$ is a prime of $E$, we denote by $f_\PP$ its absolute residue degree. 

We next recall $\ell$-adic arguments due to Katz, Ogus, and Serre. Propositions~\ref{proposition: Katz} and~\ref{proposition: Ogus} correspond to \cite[Prop. 2.7]{Ogu82}. Note that \cite[Prop. 2.7]{Ogu82} only considers the case $E=\Q$. Since the general number field case has some nontrivial additional intricacies, we will include the proofs of both propositions for the reader's convenience.

\begin{proposition}[After Katz]\label{proposition: Katz}
Let $\varrho: G_k\rightarrow \Aut(E_\lambda^d)$ be a continuous representation, unramified outside a finite set $S$ of primes of $k$, $\Oo_E$-integral, and of weight~1. Then there exists a finite extension $k'/k$, and a set $R$ of primes of $k'$ disjoint from~$S'$ such that:
\begin{enumerate}[i)]
\item Every prime $\p$ in $R$ has absolute residue degree 1. 
\item For every prime $\p$ in $R$, we have that 
\begin{equation}\label{equation: strong}
\sum_{\PP\mid \Nm(\p)} f_{\PP}v_{\PP}(a_\p) \leq \frac{[E:\Q]}{2}\,,
\end{equation}
where the sum runs over the primes $\PP$ of $E$ lying over $\Nm(\p)$ and dividing $a_\p$.
\item $R$ has density 1.
\end{enumerate}
\end{proposition}

\begin{proof}
Without loss of generality, we may assume that $\varrho$ is of the form $\varrho\colon G_k\rightarrow \GL_d(\Oo_{E_\lambda})$, where $\Oo_{E_\lambda}$ denotes the ring of integers of $E_\lambda$. Choose an integer $n\geq 1$ such that $\lambda^n$ does not divide $d$. We will consider
\begin{equation}\label{equation: reduction}
\overline \varrho\colon G_k\rightarrow \GL_d(\Oo_{E_\lambda}/\lambda^n\Oo_{E_\lambda})
\end{equation}
the reduction of $\varrho$ modulo $\lambda^n$. Choose $k'/k$ large enough so that $\overline\varrho|_{G_{k'}}$ is trivial. Let $R$ be the set of primes $\p$ of $k'$ of absolute residue degree 1, not in $S'$, and such that $\Nm(\p)>d^{2[E:\Q]}$. Let $\p$ be a prime of $R$ and denote by $p$ its absolute norm. We claim that $ii)$ is satisfied.
First note that since $\overline \varrho|_{G_k'}$ is trivial, we have
$$
a_\p\equiv d \pmod{\lambda^n}\,,
$$
which implies in particular that $a_\p$ is nonzero. We can thus write
$$
(a_\p)=\prod_{\PP|p}\PP^{v_{\PP}(a_\p)} \bb_\p
$$ 
for some nonzero ideal $\bb_\p$ of $\Oo_E$ coprime to $p$. To shorten the notation let us write $v=\sum_{\PP\mid p} f_{\PP}v_{\PP}(a_\p)$. Suppose that $v>[E:\Q]/2$. Using that $\varrho$ is of weight 1, we get
$$
\Nm(\bb_\p)= \frac{|\Nm_{E/\Q}(a_\p)|}{\prod_{\PP\mid p}\Nm(\PP)^{v_\PP(a_\p)}}\leq \left( dp^{\frac{1}{2}-\frac{v}{[E:\Q]}}\right) ^{[E:\Q]}<1\,.
$$
For the last inequality we have used that the rational number $1/2-v/[E:\Q]$ is at most $-1/(2[E:\Q])$, while we have that $p>d^{2[E:\Q]}$. But this implies that $\bb_\p$ is zero, which is a contradiction. 
\end{proof}

\begin{remark}\label{remark: general}
Note that for every $\p\in R$ such that $\Nm(\p)=p$ is unramified in $E$, we have that $a_\p\in \Oo_E$ is not divisible by $p$. Indeed, if this were the case, then $\sum_{\PP\mid p}f_{\PP} v_\PP(a_\p)\geq [E:\Q]$. Proposition \ref{proposition: Katz} thus generalizes \cite[(2.7.1)]{Ogu82}. While for $E=\Q$ the condition ``$p$~does not divide $a_p$" and \eqref{equation: strong} are equivalent, in general the latter is stronger than the statement ``$p$ does not divide $a_\p$".
\end{remark}

We will denote by $\chi_\ell\colon G_k\rightarrow E_\lambda^\times$ the $\ell$-adic cyclotomic character, giving the action of $G_k$ on the $\ell$-power roots of unity in $\Qbar$. Tate twists are normalized so that $E_\lambda(1)$ is equipped with the action of $G_k$ by $\chi_\ell$.

\begin{proposition}[After Ogus]\label{proposition: Ogus}
Let $\varrho: G_k\rightarrow \Aut(E_\lambda^d)$ be a continuous semisimple representation, unramified outside a finite set $S$ of primes of $k$, $\Oo_E$-integral, and of weight 2. Then there exists a finite extension $k'/k$ such that either 
\begin{equation}\label{equation: alternative}
\varrho|_{G_{k'}}\simeq \bigoplus_{i=1}^d E_\lambda(1)
\end{equation}
or there exists a set $R$ of primes of $k'$ disjoint from $S'$ such that:
\begin{enumerate}[i)]
\item Every prime $\p$ in $R$ has absolute residue degree 1. 
\item For every prime $\p$ in $R$, we have that $a_\p\in \Oo_E$ is not divisible by $\Nm(\p)$.
\item $R$ has a positive lower density.
\end{enumerate}
\end{proposition}

\begin{proof}
Choose an integer $n\geq 1$ such that $\ell^n>(2d)^{[E:\Q]}$.
Let $\overline \varrho$ be as in \eqref{equation: reduction}, the reduction of $\varrho$ modulo $\lambda^n$. Choose $k'/k$ large enough so that $\overline \varrho|_{G_{k'}}$ is trivial and so that $k'$ contains $\Q(\zeta_{\ell^n})$, where $\zeta_{\ell^n}$ is a root of unity of order $\ell^n$.
Let $R$ (resp. $\overline R$) be the set of primes of $k'$ of absolute degree 1, not in $S'$, and such that $a_\p$ is not divisible (resp. is divisible) by $\Nm(\p)$. Suppose that $\overline R$ has upper density 1. We will show that $\varrho|_{G_{k'}}$ is isomorphic to 
$$
\varrho_0:=\bigoplus_{i=1}^d E_\lambda(1)\,.
$$
Let $\p$ be a prime in $\overline R$, let $p$ denote its absolute norm, and let $b_\p\in \Oo_E$ be such that $a_\p=p b_\p$. Note that since $\Frob_\p|_{\Q(\zeta_{\ell^n})}$ is the identity, we have $p\equiv 1\pmod {\ell^n}$. Together with the triviality of $\overline\varrho|_{G_{k'}}$, this implies that
\begin{equation}\label{equation: bcongruence}
b_\p\equiv a_\p \equiv d \pmod{\lambda^n}\,.
\end{equation}
In particular, $\ell^n$ divides $\Nm_{E/\Q}(b_\p-d)$. Bearing in mind that $\varrho$ is of weight 2, we find
$$
|\Nm_{E/\Q}(b_\p-d)|=\prod_{\sigma\colon E\rightarrow \C}|\sigma(b_\p)-d|\leq (2d)^{[E:\Q]}\,.
$$
Together with the choice of $n$ and the above divisibility condition, this implies that $\Nm_{E/\Q}(b_\p-d)=0$. It follows that $a_\p=pd$. In other words, we have seen that
$$
\Tr(\varrho(\Frob_\p))=\Tr(\varrho_0(\Frob_\p))\qquad \text{for every $\p\in \overline R$}\,.
$$
Since $\varrho$ is semisimple and $\overline R$ has upper density 1, the Chebotarev density theorem implies that $\varrho_0\simeq \varrho|_{G_{k'}}$.  
\end{proof}

\begin{remark}
The previous proposition could be proved by reducing to the $E=\Q$ case and then invoking \cite[(2.7.2)]{Ogu82}. Indeed, it suffices to consider the representation $\Res_{E/\Q} (\varrho)$ on $\Q_\ell^{d[E:\Q]}$ obtained from $\varrho$ by regarding $E_\lambda^d$ as a $\Q_\ell$-vector space of dimension $d[E:\Q]$, and then observe that $\Tr(\Res_{E/\Q}(\varrho))=\Tr_{E/\Q}(\Tr(\varrho))$. We have included a direct proof for the reader's convenience as we will need to point to a step of the proof below.
\end{remark}

The existence of a nonzero density set of ordinary primes for an abelian surface follows from the above proposition applied to the representation attached to the action of $G_k$ on $\wedge^2 V_\ell(A)$, after noting that \eqref{equation: alternative} cannot hold for $\wedge^2 V_\ell(A)$. The latter is implied by the following criterion, which we record for later use.

\begin{proposition}\label{proposition: refinement}
Let $\varrho$ be as in Proposition \ref{proposition: Katz}, but assume further that $\varrho\simeq \wedge^2\theta$, where $\theta$ is a continuous semisimple representation unramified outside $S$, $\Oo_E$-integral, of weight~$1$, and dimension $\geq 3$. Then alternative \eqref{equation: alternative} does not occur and~$R$ can be taken of density 1.
\end{proposition}

\begin{proof}
Let $\overline R$ be as in the proof of Proposition \ref{proposition: Ogus}. It suffices to show that this set is finite. Suppose that $\p$ belongs to $\overline R$. In the proof of Proposition \ref{proposition: Ogus} we have seen that then $\Tr(\varrho(\Frob_\p))=pd$. This implies that the eigenvalues of $\Frob_\p$ acting on $\wedge^2\theta$ must all be $p$. Note that $\theta(\Frob_\p)$ has at least \emph{three}\footnote{Note that this argument amusingly fails if we cannot appeal to an auxiliary third eigenvalue. In fact, for an elliptic curve $A$ we have $\wedge^2V_\ell(A)\simeq \Q_\ell(1)$ and $\overline R$ has density $1$.} eigenvalues $\alpha$, $\beta$, $\gamma$, which must satisfy
$$
\alpha\beta=p\,,\qquad \alpha \gamma=p\,,\qquad \beta\gamma=p\,.
$$ 
These equations force $\alpha,\beta,\gamma$ to either be all $\sqrt p$ or all $-\sqrt{p}$. As a consequence, either the eigenvalues of $\theta(\Frob_\p)$ are all equal to $\sqrt p$ or all equal to $-\sqrt p$. This implies that $\pm\dim(\theta)\sqrt p=\Tr(\theta(\Frob_\p))\in \Oo_E$, which can only happen for finitely many~$p$. 
\end{proof}

The following criterion is extracted from \cite{Ser81}.

\begin{proposition}[After Serre]\label{proposition: Serre}
Let $\varrho\colon G_k\rightarrow \Aut(E_\lambda^2)$ be a continuous semisimple representation, unramified outside a finite set $S$ of primes of $k$, $\Oo_E$-integral of some weight $w\geq 1$. Suppose that $\det(\varrho)=\varepsilon \chi_\ell^w$, where $\varepsilon$ is a finite order character. If there exists a nonzero density set of primes~$\p$ of $k$ such that $a_\p=0$, then $\varrho$ is potentially abelian.  
\end{proposition}

\begin{proof}
Let $G_\lambda$ denote the Zariski closure of the image of $\varrho$ and let $\mathfrak g_\lambda$ denote the Lie algebra of $G_\lambda$. If $\varrho$ is not potentially abelian, then by \cite[Prop. 4.4]{Rib77} (which relies on \cite[Thm. 4.3]{Rib77}) we have that $\mathfrak g_\lambda$ is irreducible and nonabelian, and it thus contains $\mathfrak{sl}_2$ (note that \cite[Thm. 4.3, Prop. 4.4]{Rib77} concern the $\lambda$-adic representation attached to a modular form; however, note that the properties used in their proofs are precisely the hypotheses we make on $\varrho$ in the statement of the proposition). This implies that on no connected component of $G_\lambda$ is the trace identically equal to the zero function. This implies the proposition, as by \cite[Prop. 13, Cor. 1 to Prop. 15]{Ser81} the density of the set of primes $\p$ such that $a_\p=0$ is the proportion of connected components of $G_\lambda$ on which the trace is identically zero. 
\end{proof}

\section{$\lambda$-adic representations attached to abelian varieties} Let $E$ be a number field and let $A$ be an abelian variety defined over $k$ such that there exists a $\Q$-algebra embedding $E\hookrightarrow \End(A)\otimes \Q$. Choose a rational prime $\ell$. From the $\ell$-adic representation $\varrho_{A,\ell}\colon G_k \rightarrow \Aut(V_\ell(A))$, we will construct a $\lambda$-adic reprentation for every prime $\lambda$ of $E$ dividing $\ell$. Let $V_\lambda(A)$ denote the tensor product $V_\ell(A) \otimes_{E\otimes \Q_\ell} E_\lambda$ taken with respect to the $\Q_\ell$-algebra map
$$
E \otimes \Q_\ell\simeq \bigoplus_{\lambda' \mid \ell} E_{\lambda'} \rightarrow E_\lambda\,, 
$$
where the second map is the natural projection. Note that $V_\lambda(A)$ is an $E_\lambda$-vector space of dimension $2g/[E:\Q]$, where  $g$ denotes the dimension of $A$. Let $S$ be a finite set containing the set of primes of bad reduction for $A$.  Let $S_\ell$ denote the union of $S$ and the set of primes of $k$ lying above $\ell$. The $\Q_\ell$-linear action of $G_k$ on $V_\ell(A)$ gives rise to a $\lambda$-adic Galois representation
$$
\varrho_{A,\lambda}\colon G_k\rightarrow \Aut(V_\lambda(A))\,,
$$
which is unramified outside $S_\ell$, $\Oo_E$-integral, and of weight $1$ (see \cite[III-16]{Ser89}, \cite[Chap. II]{Rib76}, or \cite[\S0.11]{Zar89} for detailed explanations of these facts). 

Every embedding $\sigma\colon E\hookrightarrow \overline \Q_\ell$, singles out a prime $\lambda_\sigma$ of $E$ above $\ell$ such that~$\sigma$ factors via $E_\lambda$. For simplicity, we choose hereafter the prime $\ell$ to be totally split in~$E$. In this case the association of $\lambda:=\lambda_\sigma$ to $\sigma$ defines a bijection between the set of embeddings of $E$ into $\Q_\ell$ and the set of primes of $E$ over $\ell$. Whenever $E$ is Galois over $\Q$, it will be useful to regard these two sets as $\Gal(E/\Q)$-torsors. Let $V_\sigma(A)$ denote the tensor product $V_\lambda(A) \otimes_{E_\lambda} \Q_\ell$ taken with respect to the the isomorphism of $E_\lambda$ with $\Q_\ell$ determined by~$\sigma$. We equip it with an action of $G_k$ by letting this group act via $\varrho_{A,\lambda}$ on $V_\lambda(A)$ and trivially on $\Q_\ell$. We then have an isomorphism
\begin{equation}\label{equation: decTate}
V_\ell(A) \simeq \bigoplus_{\sigma\colon E\hookrightarrow \Q_\ell} V_\sigma(A)
\end{equation}
of $\Q_\ell[G_k]$-modules.
For any prime $\p$ not in $S_\ell$, we will denote by $P_\p(V_\sigma(A),T)\in \Oo_E[T]\subseteq \Q_\ell[T]$ (resp. by $a_{\p,\sigma}\in \Oo_E\subseteq \Q_\ell$) the characteristic polynomial (resp. the trace) of $\Frob_\p$ acting on $V_\sigma(A)$.
 
In case $E$ is a CM field, given an embedding $\sigma\colon E\hookrightarrow \Q_\ell$ we will denote by $\overline \sigma$ the composite of $\sigma$ with complex conjugation. Similarly, given a prime $\lambda$ of $E$ dividing~$\ell$ we will denote by $\overline \lambda$ the conjugate of $\lambda$ by complex conjugation.

As a preparation for the following section, we next observe that the mere existence of $\varrho_{A,\lambda}$ as an $\Oo_E$-integral representation is sufficient to ensure the existence of a nonzero density set $R$ of ordinary primes for $A$, whenever $A$ is an abelian variety with potential complex multiplication. 

\begin{remark}\label{remark: CM} Suppose that $A$ is an abelian variety with potential complex multiplication, that is, suppose that there exists a number field $E$ of degree $2g$ together with a $\Q$-algebra embedding $E\hookrightarrow \End(A_\Qbar)\otimes \Q$.  For the purpose of showing the existence of a nonzero density set of ordinary primes, we may assume without loss of generality that $E\hookrightarrow \End(A)\otimes \Q$. Choose a prime $\ell$ totally split in $E$. Let $R$ denote the set of primes $\p$ of $k$ outside $S$ of absolute residue degree $1$ and such that $\Nm(\p)$ is totally split in $E$. For every embedding $\sigma$ of $E$ into $\Q_\ell$, we can consider the 1-dimensional $\Q_\ell$-vector space $V_{\sigma}(A)$. In this situation, by~\eqref{equation: decTate}, the $a_{\p,\sigma}$ are the roots of $P_\p(A,T)$. From the Weil conjectures, for every $\p\in R$ we then have 
$$
a_{\p,\sigma}\cdot a_{\p,\overline \sigma}=\Nm(\p)\,.
$$
As $a_{\p,\sigma}$ and $a_{\p,\overline \sigma}$ lie in $\Oo_E$ and the rational prime $p=\Nm(\p)$ is totally split in~$E$, one among them must be a $p$-adic unit. The fact that every prime $\p$ in $R$ is ordinary for~$A$ is thus clear.
\end{remark}

In order to prove Theorems \ref{theorem: main}, \ref{theorem: main2} and \ref{theorem: main2p} we will proceed similarly as in the above remark, but we will need to additionally apply Propositions \ref{proposition: Katz},~\ref{proposition: Ogus}, and~\ref{proposition: Serre} to the representations $\varrho_{A,\lambda}$. 

\section{Proof of Theorem \ref{theorem: main}}
Suppose now that $A$ is an abelian threefold and $E$ is an imaginary quadratic field  such that there exists a $\Q$-algebra embeddding $E\hookrightarrow \End(A_\Qbar)\otimes \Q$. 
Without loss of generality we may enlarge $k$ and assume that $E\hookrightarrow \End(A)\otimes \Q$. We may  further enlarge $k$ so that it contains $E$ itself. Choose a rational prime~$\ell$ totally split in $E$, so that the two primes $\lambda,\overline \lambda$ above $\ell$ correspond to the two embeddings $\sigma,\overline\sigma$ of $E$ into $\Q_\ell$.
Let $k'$ be the finite extension of $k$ produced by Proposition~\ref{proposition: Ogus} when applied to the representation $\wedge^2\varrho_{A,\lambda}$ and the set $S_\ell$. Note that Proposition~\ref{proposition: refinement} applies, and thus alternative \eqref{equation: alternative} does not occur and the set~$R$ produced by Proposition \ref{proposition: Ogus} has density 1. We now further enlarge $k$ so that it contains $k'$. By making obvious abuses of the notation, we will regard $R$ as a set of primes of $k$.

Let $\p$ be a prime of $R$. We claim that $\p$ is ordinary. Having assumed that~$k$ contains $E$, we have that the prime $p=\Nm(\p)$ is totally split in $E$. Let $\PP$ be a prime of $E$ above $\Nm(\p)$ and let~$v$ denote an extension to $\overline \Z$ of the $\PP$-adic valuation~$v_\PP$. Let us use the notation
$$
P_\p(V_\sigma(A),T)=T^3-a_{\p,\sigma}T^2+b_{\p,\sigma}T-c_{\p,\sigma}\in \Oo_E[T]\,,
$$
and define $a_{\p,\overline\sigma},b_{\p,\overline\sigma},\dots$ similarly for $P_\p(V_{\overline\sigma}(A),T)$. In particular, we have that $b_{\p,\sigma}=\Tr(\wedge ^2 V_\sigma(A)(\Frob_\p))$. Note that 
$$
c_{\p,\sigma} \cdot c_{\p,\overline \sigma}=\det(V_\sigma(A)(\Frob_\p))\cdot\det(V_{\overline \sigma}(A)(\Frob_\p))=\det(\varrho_{A,\ell}(\Frob_\p))=p^3\,,
$$
which implies that $v(c_{\p,\sigma})+v(c_{\p,\overline \sigma})=3$. We may thus assume that $v(c_{\p,\sigma})\geq 1$ and $v(c_{\p,\overline\sigma})\geq 1$ as otherwise we are obviously done\footnote{As one can deduce from Proposition~\ref{proposition: RibetShimura}, these valuations exhibit a global behavior. This is negligible here, but it will become relevant in the proof of Theorem~\ref{theorem: main3}.}.
As $\p$ belongs to~$R$, we have that $p$ does not divide $b_{\p,\sigma}$. Therefore we cannot simultaneously have $v(b_{\p,\sigma})\geq 1$ and $v(b_{\p,\overline\sigma})\geq 1$. By symmetry between $\sigma$ and $\overline \sigma$, we may assume that $v(b_{\p,\sigma})=0$. Then the usual Newton polygon argument ensures that $P_\p(V_\sigma(A),T)$ has two roots $\alpha_\p,\beta_\p$ which are $v$-adic units. Together with the fact that $v(c_{\p,\sigma})\geq 1$, this implies that the third root $\gamma_\p$ of  
$P_\p(V_\sigma(A),T)$ must have $v$-adic valuation equal to $1$ (note that the valuation of any root $\xi_\p$ of $P_\p(A,T)$ is $\leq 1$, as $\xi_\p$, $\overline\xi_\p$ are algebraic integers satisfying $\xi_\p\overline\xi_\p=p$ and $p$ is totally split in $E$). Since complex conjugation interchanges the roots of $P_\p(V_\sigma(A),T)$ and $P_\p(V_{\overline\sigma}(A),T)$, the roots of the later polynomial are $p/\alpha_\p, p/\beta_\p,$ and $p/\gamma_\p$, which have $v$-adic valuations 1, 1 and 0.

\section{Proof of Theorem \ref{theorem: main2}}\label{section: main2}
Suppose now that $A$ is an abelian variety of dimension~$g$ and $E$ is a Galois number field of degree $g$ such that there exists a $\Q$-algebra embedding $E\hookrightarrow \End(A_\Qbar)\otimes \Q$. Without loss of generality we may enlarge $k$ and assume that $E\hookrightarrow \End(A)\otimes \Q$. We may further enlarge $k$ so that it contains $E$ itself. Choose a rational prime $\ell$ totally split in $E$. 
Let $k'$ be the finite extension of $k$ produced by Proposition~\ref{proposition: Katz} when applied to $\varrho_{A,\lambda}$ and the set $S_\ell$. Let $R$ be the set of primes it produces. Enlarge $k$ to $k'$ and regard $R$ as a set of primes of~$k$.

Let $\p$ be a prime of $R$. We claim that $u_\p(A) \geq \lceil g/2\rceil$. As $k$ contains $E$, the prime $p=\Nm(\p)$ is totally split in $E$. As the set of embeddings $\sigma$ of $E$ into $\Q_\ell$, the set of primes of~$E$ above $p$ is a $\Gal(E/\Q)$-torsor. Let $\PP$ be a prime of $E$ above $p$, and let~$v$ denote an extension to $\overline \Z$ of the $\PP$-adic valuation $v_\PP$. 
Recall that $a_{\p,\sigma}$ denotes the trace of $\Frob_\p$ acting on $V_\sigma(A)$.
By \eqref{equation: strong}, at most $g/2=[E:\Q]/2$ of the primes of~$E$ above $p$ appear in the factorization of $a_{\p,\sigma}$. From the $\Gal(E/\Q)$-torsor point of view, this translates into the assertion that the number of $\sigma$ such that the $v$-adic valuation of $a_{\p,\sigma}$ is nonzero is at most $\lfloor g/2\rfloor$. As $P_\p(A,T)=\prod_{\sigma}P_{\p}(V_{\sigma}(A),T)$, the theorem follows.

\section{Proof of Theorem \ref{theorem: main2p}} By successively replacing $k$ by $K$ and by taking the compositum with a quadratic field, we may assume that $E\hookrightarrow \End(A)\otimes \Q$ and that $kE/k$ is Galois of degree~$3$. Let $R$ be the set of primes of $k$ of absolute degree~$1$ which are inert in $kE$. Take a rational prime $\ell$ that splits completely in the Galois closure $L$ of $E$, let~$\lambda$ be a prime of $E$ lying above $\ell$, and let $\sigma$ be the associated embedding of $E$ into~$\Q_\ell$. By \cite[Lem. 4.5.1]{Rib76}, we have $\det(\varrho_{A,\lambda})=\chi_\ell$, and in particular $\varrho_{A,\lambda}$ satisfies the hypotheses of Proposition \ref{proposition: Serre}. Let $\p$ be a prime in $R$, let $p$ denote its absolute norm $\Nm(\p)$, and let $\PP$ denote the prime of $E$ lying above $p$. Let $v$ denote an extension to $\overline \Z$ of the $\PP$-adic valuation~$v_\PP$. We claim that $v(a_{\p,\sigma})$ is independent of the choice of~$\sigma$. Indeed, let~$\sigma_0$ be an embedding of~$L$ into $\Q_\ell$ extending~$\sigma$. Any embedding of $E$ into~$\Q_\ell$ is obtained by precomposing $\sigma_0$ with the inclusion $E\subseteq L$ and an element of the order 3 subgroup $H$ of $\Gal(L/E)$. The claim then follows from the fact that $H$ is the stabilizer in $\Gal(L/E)$ of the primes of~$L$ above~$\PP$. Therefore, if $v(a_{\p,\sigma})=0$ for every~$\p$ in a nonzero density subset of~$R$, then the equality $P_\p(A,T)=\prod_{\sigma}P_{\p}(V_{\sigma}(A),T)$ implies the theorem. Otherwise, we have that $\PP$ divides the ideal $(a_{\p,\sigma})$ for every~$\p$ in a nonzero density subset $R'$ of~$R$. By comparing the norms of the two ideals, we immediately see that this implies that $a_{\p,\sigma}=0$ for every $\p$ in $R'$. By Proposition~\ref{proposition: Serre}, we have that $\varrho_{A,\lambda}$ is potentially abelian. We may then derive the existence of a nonzero density set of primes by applying the argument of Remark~\ref{remark: CM} (in fact, in this case $A$ is $\Qbar$-isogenous to the power of an abelian variety with complex multiplication).

\section{Hodge-Tate theoretic input} Return to the general setting in which $A$ is an abelian variety defined over $k$ of dimension $g$ and $E$ a Galois number field that embeds into $\End(A)\otimes \Q$. Throughout this section we will assume that~$k$ contains $E$, that the extension $k/E$ is Galois, and that $\ell$ is a rational prime totally split in $k$.

Let $\lambda$ be a prime of $E$ lying above~$\ell$. As $\varrho_{A,\ell}$ has the Hodge--Tate property, so has it $\varrho_{A,\lambda}$ by \cite[Prop. 1.5.3]{Rib76}. Therefore $\delta_\lambda=\det(\varrho_{A,\lambda})$ also has the Hodge--Tate property, and from \cite[p. 761]{Rib76} one deduces that $\delta_\lambda$ is associated with an algebraic Hecke character $\delta$ of type $A_0$ with values in $E$ of the field $k$ in the sense of Weil \cite{Wei55}. Via global class field theory, we may think of $\delta$ as a homomorphism from the group of fractional ideals of $k$ coprime to some integral ideal $\Mm$ of $\Oo_k$, a modulus for $\delta$. In particular, for every embedding $\sigma$ of $E$ into $\C$, there exists an integer $n_\sigma$ such that
$$
\delta((\alpha))=\prod_{\sigma\colon E\hookrightarrow \C} \sigma(\alpha)^{n_\sigma[k:E]}
$$
for every $\alpha$ in $E^\times$ such that $\alpha\equiv^{\times} 1\pmod \Mm$, that is, for every $\alpha\in E^\times$ which is multiplicatively congruent to $1$ modulo $\Mm$. Note that the expression $\sum_{\sigma} n_\sigma \sigma\circ \Nm_{k/E}$ gives the infinity type of the Hecke character $\delta$. 
Loosely speaking, we will also refer to $(n_\sigma)_\sigma$ as the infinity type of~$\delta_\lambda$. 

As we used in the Proof of Theorem \ref{theorem: main2p}, by \cite[Lem. 4.5.1]{Rib76}, when $E$ is totally real and has degree equal to $g$, then $\delta_\lambda=\chi_\ell$. The purpose of this section is to describe in general the infinity type $(n_\sigma)_\sigma$. Fix an embedding $\iota\colon k \hookrightarrow \C$. 

\begin{definition}\label{definition: signature} For an embedding $\sigma$ of $E$ into $\C$, let
$$
r_\sigma=\dim_{\C}\big(H^0(A,\Omega^1_{A/k})\otimes_{k\otimes E,\iota\otimes\sigma} \C\big)
$$
denote the ``multiplicity of the action of $E$ via $\sigma$'' on $H^0(A_\C,\Omega^1_{A_\C/\C})$. Note that $\sum_\sigma r_\sigma=g$. We refer to $(r_\sigma)_\sigma$ as the signature of the pair $(A,E)$.
\end{definition}

Via the embedding $\iota$ and the inclusion $E\subseteq K$, we may freely identify the set of embeddings of $E$ into $\C$ with the Galois group $\Gal(E/\Q)$: for $\sigma \in \Gal(E/\Q)$, we will still denote by $\sigma$ the composition $\iota \circ \sigma$.

Let $\Ll$ be a prime of $k$ lying above~$\lambda$, and let $\C_\ell$ denote the completion of $\overline{ k_\Ll}$. We fix an embedding $\kappa: \C \hookrightarrow \C_\ell$ such that the composition $\kappa \circ \iota: k \hookrightarrow \C_\ell$ coincides with the natural inclusion $k\subseteq k_{\Ll}\subseteq \C_\ell$. In virtue of this compatibility, by abuse of language, we still denote by $\iota$, the natural inclusion $k\subseteq k_{\Ll}\subseteq \C_\ell$.
Even more, given $\sigma \in \Gal(E/\Q)$, we will still denote by~$\sigma$ the embedding obtained by composing $\sigma$ with any of the inclusions $E\subseteq k \subseteq k_{\Ll}\subseteq \C_\ell$. 

We may summarize our series of abuses of notation by saying that, for $\sigma \in \Gal(E/\Q)$, any composition of maps in the commutative diagram
$$
\xymatrix{
& & & & \C \ar[d]^{\kappa} \\
E\ar[r]^{\sigma} & E\ar[r] & k \ar[urr]^{\iota}\ar[r] & k _{\Ll}\ar[r]  & \C_\ell}
$$
departing from the leftmost occurrence of $E$ will still be called $\sigma$.

\begin{proposition}\label{proposition: RibetShimura}
The inverse infinity type $(n_{\sigma^{-1}})_\sigma$ and the complex conjugate signature $(r_{\overline\sigma})_\sigma$ coincide.
\end{proposition}

\begin{proof}
 We depart from the isomorphism
\begin{equation}\label{equation: Tateiso}
V_\ell(A)(-1)\otimes_{\Q_\ell} \C_\ell\simeq H^1(A,\Oo_A)\otimes_{k,\iota} \C_\ell\oplus H^0(A,\Omega^1_{A/k})\otimes_{k,\iota}\C_\ell(-1) 
\end{equation}
proven by Tate \cite[Cor. 2, p. 180]{Tat67}. As in the discussion preceeding the proposition, let $\sigma\colon E\subseteq k\hookrightarrow k_\Ll$ denote the embedding associated to the prime $\Ll$. It induces a $\C_\ell$-algebra map $\C_\ell\otimes E\rightarrow \C_\ell$, which by abuse of notation we still denote by $\sigma$. We now take the tensor $\otimes_{\C_\ell\otimes E,\sigma}\C_\ell$ of \eqref{equation: Tateiso} with respect to this map, obtaining
$$
V_\sigma(A)(-1)\otimes_{\Q_\ell} \C_\ell\simeq H^1(A, \Oo_A)\otimes_{k\otimes E,\iota\otimes \sigma} \C_\ell \oplus H^0(A,\Omega^1_{A/k})\otimes_{k\otimes E,\iota\otimes \sigma} \C_\ell(-1)\,.
$$ 
Taking first the tensor of the above expression with $\C_\ell(1)$, then $G_{k_\Ll}$-invariants, and then the tensor with $\C_\ell$ again, \cite[Thm. 1]{Tat67} yields
\begin{equation}\label{equation: TateRibet}
(V_\sigma(A)\otimes_{\Q_\ell}\C_\ell)^{G_{k_\Ll}}\otimes \C_\ell\simeq H^0(A,\Omega^1_{A/k})\otimes_{k\otimes E,\iota \otimes \sigma} \C_\ell\,.
\end{equation}
Note that $V_\sigma(A)$, as a subspace of $V_\ell(A)$, has the Hodge--Tate property, and thus we have a noncanonical isomorphism 
\begin{equation}\label{equation: TateRibet2}
V_\sigma(A)\otimes_{\Q_\ell} \C_\ell\simeq \C_\ell(1)^{m_\sigma}\oplus \C_\ell^{s_{\sigma}}
\end{equation}
for certain integers $m_\sigma,s_\sigma \geq 0$. Therefore $\delta_\lambda\otimes_{\Q_\ell}\C_\ell|_{G_{k_{\Ll}}}$ is noncanonically isomorphic to $\C_\ell(m_\sigma)$. 

By \cite[III-30,III-44, III-48]{Ser89}, we have that $n_{\sigma^{-1}}=m_\sigma$. On the other hand, the comparison of \eqref{equation: TateRibet} and \eqref{equation: TateRibet2} imply that $r_\sigma=s_\sigma$. It thus only remains to show that $m_{\sigma}=s_{\overline \sigma}$. But this immediately follows from the isomorphism
$$
V_\sigma(A)\simeq V_{\overline \sigma}(A)^{\vee}(1)\,,
$$
where $V_{\overline\sigma}(A)^{\vee}$ denotes the dual of $V_{\overline\sigma}(A)$ (see the proof of \cite[Prop. 3.4]{Rib92}, which extends \emph{verbatim} to our more general setting).
\end{proof}

\begin{remark}
When $A$ is an abelian variety with CM, the $\delta_\lambda$ are the Hecke characters that give the $L$-function of $A$. In this particular case, Proposition \ref{proposition: RibetShimura} is the restatement that the infinity type of $A$ is the ``inverse CM type" of $A$. 
\end{remark}

\section{Proof of Theorem \ref{theorem: main3}} Suppose that $A$ is an abelian fourfold and $E$ an imaginary quadratic field that embeds into $\End(A_\Qbar)\otimes \Q$. Suppose that the pair $(A,E)$ has signature $(2,2)$. Choose a rational prime $\ell$ totally split in $E$, and let $\lambda,\overline\lambda$ denote the corresponding primes of $E$. Let $\sigma,\overline \sigma$ denote the embeddings of $E$ into $\Q_\ell$ attached to $\lambda,\overline\lambda$.
 We may enlarge $k$ so that it contains $E$, so that $E$ actually embeds in $\End(A)\otimes \Q$, and so that the field produced by Proposition \ref{proposition: Ogus} when applied to $\wedge^2\varrho_{A,\lambda}$ and $S_\ell$ is contained in $k$. We thus regard $R$ as set of primes of $k$, which by Proposition \ref{proposition: refinement} we can assume of density 1. We finally replace $k$ by its normal closure over $E$.

Let $\p$ be a prime of $R$. We claim that it is ordinary. Note that $p=\Nm(\p)$ is totally split in $E$. Let $\PP$ be a prime of $E$ above $\Nm(\p)$ and let $v$ denote an extension to $\overline \Z$ of the $\PP$-adic valuation $v_\PP$. Let us denote 
$$
b_{\p,\sigma}=\Tr(\wedge ^2V_\sigma(A)(\Frob_\p))\,,\qquad \delta_{\p,\sigma}:=\det(V_\sigma(A)(\Frob_\p))\,.
$$
As in the proof of Theorem \ref{theorem: main2}, one sees that $\delta_{\p,\sigma}\cdot\delta_{\p,\overline\sigma}=p^4$. We next claim that the hypothesis that the pair $(A,E)$ has signature $(2,2)$ implies that $v(\delta_{\p,\sigma})=v(\delta_{\p,\overline \sigma})=2$. By Proposition~\ref{proposition: RibetShimura}, the infinity type of $\delta_\lambda$ is $(2,2)$. Let $\Mm$ denote the modulus of~$\delta_\lambda$ and let $h$ denote the product of the class number of $k$ and the size of $(\Oo_k/\Mm)^\times$. Then there exists $\beta_\p$ in $E$ such that
$$
\Nm_{k/E}(\p)^h=(\beta_\p)\,,\quad  \beta_\p\equiv^{\times} 1\pmod \Mm\,,\quad\text{and}\quad \delta_{\p,\sigma}^h=\sigma(\beta_\p)^{2h}\overline\sigma(\beta_\p)^{2h}\,.
$$ 
The last equality implies that $v(\delta_{\sigma,\p})=v(\delta_{\overline\sigma,\p})$, which shows the claim.

As in the proof of Theorem~\ref{theorem: main2}, the fact that $\p$ belongs to~$R$ implies that we cannot simultaneously have $v(b_{\p,\sigma})\geq 1$ and $v(b_{\p,\overline\sigma})\geq 1$. By symmetry between $\sigma$ and $\overline \sigma$, we may assume that $v(b_{\p,\sigma})=0$. Then the usual Newton polygon argument ensures that $P_\p(V_\sigma(A),T)$ has two roots which are $v$-adic units. Since $v(\delta_{\p,\sigma})=2$ and the $v$-adic valuation of a root is at most $1$, the other two roots of  
$P_\p(V_\sigma(A),T)$ must have $v$-adic valuation equal to $1$. But this concludes the proof, as complex conjugation, which sends a root $\alpha_\p$ to $p/\alpha_\p$, interchanges the roots of $P_\p(V_\sigma(A),T)$ and $P_\p(V_{\overline\sigma}(A),T)$.

\begin{remark}
By \cite{Shi63} abelian fourfolds defined over a number field $k$ with multiplication by an imaginary quadratic field of signature $(2,2)$ exist. In fact, if $k$ has a real place, then as argued in \cite[Prop. 3.4]{Gui12} any abelian fourfold defined over~$k$ with multiplication by an imaginary quadratic field must have signature $(2,2)$.
\end{remark} 

\begin{remark}
Let $\zeta_3$ denote a primitive cubic root of unity. A source of examples of simple abelian threefolds with imaginary multiplication by $\Q(\zeta_3)$ are the Jacobians of Picard curves defined over a number field~$k$ containing $\Q(\zeta_3)$. These are curves $C$ of the form $y^3=f(x)$, where $f(x)$ is a separable polynomial in $k[x]$ of degree~4. By considering the action of $\Q(\zeta_3)$ on the basis of the regular differentials $dx/y,dx/y^2,xdx/y^2$ of $C$, one immediately finds that the signature is $(2,1)$ (for a generic Picard curve, see \cite{AFP21} for the determination of the value of the density of ordinary primes). By replacing $f(x)$ by a separable degree 5 polynomial, we obtain examples of abelian fourfolds with imaginary multiplication by $\Q(\zeta_3)$. Note however that Theorem \ref{theorem: main3} does not apply in this case, as the signature is $(3,1)$.
\end{remark}

\section{A Tate module tensor decomposition} Let $A$ be an abelian variety defined over~$k$ of dimension $g\geq 1$. Let $D$ be a simple algebra over $\Q$; denote by $K$ the center of $D$ and by~$t$ the Schur index of $D$. Suppose that $D$ is a subalgebra of $\End(A)\otimes \Q$ and let $E$ be a maximal subfield of $D$. Recall that $E$ has degree~$t$ over~$K$. Let~$\ell$ be a rational prime totally split in $E$, and let $\lambda$ be a prime of~$E$ above~$\ell$. As $D\otimes_K E_\lambda$ splits as $\M_t(E_\lambda)$, and acts on $V_\ell(A)\otimes_{K\otimes \Q_\ell} E_\lambda$ compatibly with $G_k$, we find a $G_k$-equivariant isomorphism
$$
V_\ell(A)\otimes_{K\otimes \Q_\ell} E_\lambda\simeq \bigoplus_{i=1}^t V\,,
$$
where $V$ is an $E_\lambda$-vector space of dimension $2g/(t[K:\Q])$. More canonically, by considering the obvious natural map, one sees that this isomorphism takes the form
\begin{equation}\label{equation: tensordec}
\big(D\otimes_E E_\lambda \big) \otimes_{E_\lambda} V_\lambda(A) \simeq V_\ell(A)\otimes_{K\otimes \Q_\ell} E_\lambda\,,
\end{equation}
where $V_\lambda(A)=V_\ell(A)\otimes_{E\otimes \Q_\ell} E_\lambda$. In the above expression, the tensor products $V_\lambda(A)$ and $V_\ell(A)\otimes_{K\otimes \Q_\ell} E_\lambda$ are taken with respect to compatibly chosen $\Q_\ell$-algebra maps $E\otimes \Q_\ell\rightarrow E_\lambda$ and $K\otimes \Q_\ell\rightarrow E_\lambda$.
Note that under the assumption that $D$ is contained in $\End(A)\otimes \Q$, the action of $G_k$ on $D$ is trivial. With the appropriate modifications, this isomorphism can be promoted to an isomorphism also in the more general situation where one only requires that $D$ be contained in $\End(A_\Qbar)\otimes \Q$ and $K$ be contained in $\End(A)\otimes \Q$ (this can be achieved as in \cite[Prop. 2.8]{FG19}, but we will not need it here).

\begin{remark}\label{remark: quatabfourfold}
It was mentioned in the Introduction that if $A$ is an abelian fourfold such that $\End(A_\Qbar)\otimes \Q$ is isomorphic to a quaternion algebra $D$ over $\Q$ and $E$ is a quadratic subfield of $D$, then $(A,E)$ has signature $(2,2)$. One way to see this is by using Proposition \ref{proposition: RibetShimura}, decomposition \eqref{equation: tensordec} with $K=\Q$, and noting that $\det(\varrho_{A,\ell})=\chi_\ell^2$.
\end{remark}
 
\section{Proof of Corollary \ref{corollary: main 5}} We may assume that $A$ is an abelian variety such that $\End(A)\otimes \Q$ contains a simple subalgebra~$D$ as above, where $K$ is now an imaginary quadratic field. Note that the action of $D$ on the $2g$-dimensional vector space $H^1(A_\C,\Q)$ imposes the divisibility condition $t^2\mid g$, which, together with the imposed constraints on~$t$, bounds the dimension $g$ as in the statement of the corollary. 

Let $\tau,\overline \tau$ denote the two embeddings of $K$ into $\Q_\ell$. We may assume that $\tau$ is the restriction to $K$ of the embedding $\sigma\colon E\hookrightarrow \Q_\ell$ associated to $\lambda$. Let $\overline\sigma\colon E\hookrightarrow \Q_\ell$ denote an embedding restricting to $\overline \tau$. Write $V_\tau(A)=V_\ell(A)_{K\otimes \Q_\ell,\tau}\Q_\ell$ and similarly for $V_{\overline \tau}(A)$.
From the isomorphism $V_\ell(A)\simeq V_{\tau}(A)\oplus V_{\overline \tau}(A)$ and decomposition \eqref{equation: tensordec}, for every prime $\p$ outside $S_\ell$, we find
\begin{equation}\label{equation: polypow}
P_\p(A,T)=P_\p(V_{\overline \tau}(A),T)\cdot P_\p(V_\tau(A),T)=\big(P_\p(V_\sigma(A),T)\cdot P_\p(V_{\overline\sigma}(A),T)\big)^t\,.
\end{equation}
Note that the polynomials $P_\p(V_{\sigma}(A),T)$ and $P_\p(V_{\overline\sigma}(A),T)$, which a priori have coefficients in $\Oo_E$, have in fact coefficients in $\Oo_K$. Moreover, their product $h_\p$ has coefficients in $\Z$. 

Note that $\deg(h_\p)=2g/t$. Therefore, when $t=g/2$, the corollary follows by applying Proposition \ref{proposition: Ogus} exactly in the same way as when one asserts the existence of a nonzero density set of ordinary primes for an abelian surface.  

When $t=g/3$, identity \eqref{equation: polypow} reduces the proof of the corollary to that of Theorem~\ref{theorem: main}. When $t=g/4$ and $(A,K)$ has signature $(g/2,g/2)$, Proposition \ref{proposition: RibetShimura} and decomposition \eqref{equation: tensordec} imply that $\det(\varrho_{A,\lambda})$ has infinity type $(2,2)$, and then the proof of the corollary reduces to that of Theorem \ref{theorem: main3}. 

\begin{remark}\label{remark: final}
From Theorem \ref{theorem: main3}, we deduced in the Introduction the abundance of ordinary primes for an abelian fourfold $A$ such that $\End(A_\Qbar)\otimes \Q$ is isomorphic to a quaternion algebra over $\Q$. An alternative proof of this fact can be obtained by using the method employed in the proof of Corollary~\ref{corollary: main 5}, but taking $K=\Q$ (which essentially reduces the question to the case of abelian surfaces).  
\end{remark}
\subsection*{Acknowlegements.} Thanks to Bjorn Poonen for explaining the content of \cite[Prop. 2.7.1]{Ogu82} to the attendants of the MIT number theory group meeting. Thanks to Drew Sutherland for driving my attention to the question considered in this note in the case of Picard curves. Thanks to Xavier Guitart for helpful discussions and for his comments and corrections to a previous version of this manuscript, as well as for alerting me of the existence of \cite{Suh20}. I was financially supported by the Simons Foundation grant 550033.


\begin{thebibliography}{McK-Sta}
\bibitem[AFP21]{AFP21} S. Asif, F. Fit\'e, and D. Pentland, \emph{Computing L-polynomials of Picard curves from Cartier--Manin matrices. With an appendix by Andrew Sutherland}, arXiv:2010.07247v2 (2021).

\bibitem[BCGP18]{BCGP18} G. Boxer, F. Calegary, T. Gee, and V. Pilloni, \emph{Abelian surfaces over totally real fields are potentially modular}, available on arXiv:1812.09269.

\bibitem[FG19]{FG19} F. Fit\'e, X. Guitart, \emph{Tate module tensor decompositions and the Sato--Tate conjecture for certain abelian varieties potentially of $\GL_2$-type}, available on arXiv:1909.11712.

\bibitem[Gui12]{Gui12} X. Guitart, \emph{Abelian varieties with many endomorphisms and their absolutely simple factors}, Rev. Mat. Iberoam., \textbf{28}(2): 591--601, 2012.

\bibitem[Mum69]{Mum69} D. Mumford, \emph{A note on Shimura's paper ``Discontinuous groups and abelian varieties"}, Math. Ann. \textbf{181} (1969), 345--351. 

\bibitem[Mum74]{Mum74} D. Mumford, \emph{Abelian varieties}, volume \textbf{5} of Tata Institute of Fundamental Research Studies in Mathematics. Published for the Tata Institute of Fundamental Research,Bombay; by Hindustan Book Agency, New Delhi, 2008. With appendices by C. P. Ramanujam and Yuri Manin, Corrected reprint of the second (1974) edition.

\bibitem[Noo95]{Noo95} R. Noot, \emph{Abelian varieties-Galois representations and properties of ordinary reduction}, Compositio Mathematica \textbf{97}, n. 1-2 (1995), p. 161--171.

\bibitem[Ogu82]{Ogu82} A. Ogus, \emph{Hodge cycles and crystalline cohomology}. In \emph{Hodge cycles, motives, and Shimura varieties}, Springer Lecture Notes in Mathematics \textbf{900}, Springer-Verlag, 1982.

\bibitem[Pin98]{Pin98} R. Pink, \emph{$\ell$-adic algebraic monodromy groups, cocharacters, and the Mumford--Tate conjecture}, J. reine angew. Math. \textbf{495} (1998), 187--237.

\bibitem[Rib76]{Rib76} K.A. Ribet, \emph{Galois action on division points on abelian varieties with many real multiplications}, Am. J. Math. \textbf{98} (1976), 751--804.

\bibitem[Rib77]{Rib77} K.A. Ribet, \emph{Galois Representations attached to Eigenforms with Nebentypus}, Lect. Notes in Math., \textbf{601}, Springer-Verlag, 1977, 17--52.

\bibitem[Rib92]{Rib92} K.A. Ribet, \emph{Abelian varieties over $\Q$ and modular forms},
Algebra and topology 1992 ({T}aej\u on), Korea Adv. Inst.
Sci. Tech. (1992), 53--79.

\bibitem[Saw16]{Saw16}  W. Sawin, \emph{Ordinary Primes for Abelian Surfaces, Comptes Rendus Mathematique} \textbf{354}, No.  6 (2016).

\bibitem[Ser81]{Ser81} J.-P. Serre, \emph{Quelques applications du th\'eor\`eme de densit\'e de Chebotarev}, Inst. Hautes \'Etudes Sci. Publ. Math. No. \textbf{54} (1981), 323–401.

\bibitem[Ser89]{Ser89} 
J.-P. Serre, \emph{Abelian $\ell$-adic Representations and Elliptic Curves}, Addison--Wesley Publ. Co., Reading, MA, 1989.

\bibitem[Shi63]{Shi63} G. Shimura, \emph{On analytic families of polarized abelian varieties and automorphic functions}, Ann. of Math. \textbf{78}, No. 1, July, 1963.

\bibitem[Shi98]{Shi98} G. Shimura, \emph{Abelian varieties with complex multiplication and modular functions},
Princeton Mathematical Series, vol. \textbf{46}, Princeton University Press, Princeton, NJ.

\bibitem[Suh20]{Suh20} J. Suh, \emph{Ordinary primes in Hilbert modular varieties}, Compositio Math. \textbf{156} (2020), 647--678.

\bibitem[Tat67]{Tat67} J. Tate, \emph{$p$-divisible groups} in Proc. Conf. on Local Fields (Driebergen), Springer-Verlag, 1967, pp. 148--183.

\bibitem[Wei55]{Wei55} A. Weil, \emph{On a certain type of characters of the id\`ele-class group of an algebraic number-field}. Proceedings of the international symposium on algebraic number theory, Tokyo \& Nikko, 1955, pp. 1--7. Science Council of Japan, Tokyo, 1956. 

\bibitem[Wil88]{Wil88} A. Wiles, \emph{On ordinary $\lambda$-adic representations associated to modular forms}, Invent. Math. \textbf{94} (1988), 529--573.

\bibitem[Zar89]{Zar89} Y. Zarhin, \emph{Torsion of abelian varieties over $\GL(2)$-extensions of number fields}, Math, Ann. 284, 631--646 (1989).
\end{thebibliography}
\end{document}